%% file: improved_right_tail_asymptotics_and_large_deviations.tex
\documentclass[11pt,reqno,tbtags,draft]{amsart}
\usepackage{amssymb}
\usepackage{url}
\usepackage[square,numbers]{natbib}
\bibpunct[, ]{[}{]}{;}{n}{,}{,}
\title[Improved QuickSort right-tail upper bounds]
{%Refined upper bounds on the right tail of the limiting {\tt QuickSort} distribution
QuickSort:\ Improved right-tail asymptotics for the limiting distribution, \\ and large deviations}
\newcommand\urladdrx[1]{{\urladdr{\def~{{\tiny$\sim$}}#1}}}
\author{James Allen Fill}
\address{Department of Applied Mathematics and Statistics,
The Johns Hopkins University,
3400 N.~Charles Street,
Baltimore, MD 21218-2682 USA}
\email{jimfill@jhu.edu}
\urladdrx{http://www.ams.jhu.edu/~fill/}
\thanks{Research of both authors supported by 
the Acheson~J.~Duncan Fund for the Advancement of Research in
Statistics.}

\author{Wei-Chun Hung}
\address{Department of Applied Mathematics and Statistics,
The Johns Hopkins University,
3400 N.~Charles Street,
Baltimore, MD 21218-2682 USA}
\email{whung6@jhu.edu}

\numberwithin{equation}{section}

\allowdisplaybreaks

\theoremstyle{plain}% default
\newtheorem{theorem}{Theorem}[section]
\newtheorem{lemma}[theorem]{Lemma}
\newtheorem{proposition}[theorem]{Proposition}

\theoremstyle{definition}

\newtheorem{remark}[theorem]{Remark}

\theoremstyle{remark}

\newenvironment{romenumerate}[1][-10pt]{% optional argument changes indentation
\addtolength{\leftmargini}{#1}\begin{enumerate}% gives (i), (ii) etc.
 }{\end{enumerate}}

\newcounter{oldenumi}
{\setcounter{oldenumi}{\value{enumi}}
\begin{romenumerate} \setcounter{enumi}{\value{oldenumi}}}
{\end{romenumerate}}

\newcounter{thmenumerate}

\newcounter{xenumerate}   %no left indentation; thus wider lines

\newcommand{\refT}[1]{Theorem~\ref{#1}}

\newcommand{\refL}[1]{Lemma~\ref{#1}}
\newcommand{\refR}[1]{Remark~\ref{#1}}
\newcommand{\refS}[1]{Section~\ref{#1}}

\newcommand\marginal[1]{\marginpar{\raggedright\parindent=0pt\tiny #1}}
\newcommand\REM[1]{{\raggedright\texttt{[#1]}\par\marginal{XXX}}}

\begingroup
  \divide\count255 by 60
  \multiply\count255 by -60
  \advance\count255 by \time
  \ifnum \count255 < 10 \xdef\klockan{\the\count1.0\the\count255}
  \else\xdef\klockan{\the\count1.\the\count255}\fi
\endgroup

\newcommand\noqed{\renewcommand{\qed}{}} % for proof with explicit \qed

\def\rompar(#1){\textup(#1\textup)}    % usage: \rompar(...)

\def\xexp(#1){e^{#1}}

\newcommand\bbR{\mathbb R}

\newcounter{CC}
\newcounter{cc}
\renewcommand\L{\operatorname{L}}

\newcommand\Var{\operatorname{Var}}

\newcommand\Fbar{\overline F}
\newcommand\Fu{\underbar{$F$}}

\newcommand{\mgf}{moment generating function}

\newcommand\hpsi{\widehat \psi}

\newcommand\hZ{\widehat Z}

\input quicksort_density_tails_def.tex

\newcommand{\ignore}[1]{}

\newcommand{\ro}[1]{\uppercase\expandafter{\romannumeral #1}}

\begin{document}

\maketitle

\vspace{-.3in}
\begin{center}
January~18, 2019
\end{center}
\vspace{.1in}

\begin{abstract}
We substantially refine asymptotic logarithmic upper bounds
%---extending from 
%%only the lead 
%one term to 
%%three terms
%three---
produced by Svante Janson (2015) on the right tail of the limiting {\tt QuickSort} distribution function~$F$ and by Fill and Hung (2018) on the right tails of the corresponding density~$f$ and of the absolute derivatives of~$f$ of each order.  
%All our results 
%match two-term lower bounds for the functions in question to two terms
%and
%match conjectured asymptotic expansions to three terms.
For example, we establish an upper bound on $\log[1 - F(x)]$ that matches conjectured asymptotics of Knessl and Szpankowski (1999) through terms of order $(\log x)^2$; the corresponding order for the Janson (2015) bound is the lead order, $x \log x$. 

Using the refined asymptotic bounds on~$F$, we derive right-tail large deviation (LD) results for the distribution of the number of comparisons required by {\tt QuickSort} that 
%sharpen somewhat
substantially sharpen 
the two-sided LD results of McDiarmid and Hayward (1996).
\end{abstract}

\medskip
\par\noindent
\begin{small}
{\bf Keywords:\ 	}{\tt QuickSort}, asymptotic bounds, tails of distributions, large deviations, moment generating functions, Chernoff bounds
%, Landau--Kolmogorov inequality
\end{small}

%\newpage
    
\section{Introduction}
\label{S:intro}
To set the stage, and for the reader's convenience, we repeat here relevant portions of Section~1 of Fill and Hung~\cite{fill2019density_tails}.
Let $X_{n}$ denote the (random) number of comparisons when sorting~$n$ distinct numbers using the algorithm {\tt QuickSort}. Clearly $X_0 = 0$, and for $n \geq 1$ we have the recurrence relation
\[
X_{n} \overset{\mathcal{L}}{=} X_{U_{n} - 1} + X^{*}_{n-U_{n}} + n - 1,
\]
where $\overset{\mathcal{L}}{=}$ denotes equality in law (i.e.,\ in distribution); 
$X_{k} \overset{\mathcal{L}}{=} X^{*}_{k}$; the random variable $U_{n}$ is uniformly distributed on 
$\{1,\dots,n\}$; and $U_{n}, X_{0}, \dots , X_{n-1}$, $X^{*}_{0}, \dots , X^{*}_{n-1}$ are all independent. It is well known that 
\[
\mu_n := \mathbb{E}X_{n} = 2\left(n+1\right)H_{n}-4n,
\] where $H_{n}$ is the $n$th harmonic number $H_n := \sum_{k=1}^{n} k^{-1}$ and (from a simple exact expression) that $\Var X_n = (1 + o(1)) (7 - \frac{2 \pi^2}{3}) n^2$. To study distributional asymptotics, we first center and scale $X_{n}$ as follows:
\begin{equation}
\label{Zn}
Z_{n} = \frac{X_{n}-\mu_n}{n}.
\end{equation}
Using the Wasserstein $d_{2}$-metric, R\"osler \cite{rosler1991limit} proved that $Z_{n}$ converges to $Z$ weakly as $n \rightarrow \infty$.  Using a martingale argument, R\'egnier \cite{regnier1989limiting} proved that the slightly renormalized $\frac{n}{n + 1} Z_n$ converges to $Z$ in $L^{p}$ for every finite~$p$, and thus in distribution; equivalently, the same conclusions hold for $Z_n$.  The random variable~$Z$ has everywhere finite moment generating function with $\mathbb{E}Z = 0$ and 
$\Var Z = 7-\left(2\pi ^{2}/3\right)$. Moreover, $Z$ satisfies the distributional identity
\begin{equation}
\label{distid}
Z \overset{\mathcal{L}}{=} U Z + (1-U) Z^* + g(U).
\end{equation}
On the right, $Z^* \overset{\mathcal{L}}{=} Z$; $U$ is uniformly distributed on $\left(0,1\right)$; $U, Z, Z^*$ are independent; and 
\[
g(u) := 2u \ln u + 2 (1-u) \ln (1-u) + 1.
\]
Further, the distributional identity together with the condition that $\mathbb{E}Z$ (exists and) vanishes characterizes the limiting {\tt Quicksort} distribution; this was first shown by 
R\"osler~\cite{rosler1991limit} under the additional condition that $\Var Z < \infty$, and later in full by Fill and Janson~\cite{fill2000fixedpoints}.

Fill and Janson \cite{fill2000smoothness} derived basic properties of the limiting {\tt QuickSort} distribution $\mathcal{L}(Z)$.  In particular, they proved that $\mathcal{L}(Z)$ has a (unique) continuous density~$f$ which is everywhere positive and infinitely differentiable.
%,and for every $k \geq 0$ that $f^{(k)}$ is bounded and enjoys superpolynomial decay in both tails, 
%that is, for each 
%$p \geq 0$ and $k \geq 0$ there exists a finite constant $C_{p, k}$ such that 
%$\left| f^{(k)}(x) \right| \leq C_{p,k} |x|^{-p}$ for all $x \in \mathbb{R}$.

Janson~\cite{janson2015tails} studied logarithmic asymptotics in both tails for the corresponding distribution function~$F$, and Fill and Hung~\cite{fill2019density_tails} did the same for~$f$ and each of its derivatives.  For right tails, all these results can be summarized in the following theorem.
We let $\Fbar(x) := 1 - F(x)$, and for a function $h:\bbR\to\bbR$ we write
\begin{equation}
\label{xnorm}
\|h\|_x := \sup_{t \geq x} |h(t)|.
\end{equation}
 
\begin{theorem}[\cite{janson2015tails}, Thm.~1.1; \cite{fill2019density_tails}, Thms.~1.1--1.2]
\label{T:oldmain}
\ \\
{\rm (a)}~As $x \rightarrow \infty$, the limiting {\tt QuickSort} density function~$f$ satisfies
\begin{equation}
\label{oldright}
\exp [-x\ln x - x\ln \ln x + O(x)] 
\leq f(x) \leq \exp [- x \ln x + O(x)].
\end{equation}
{\rm (b)}~Given an integer $k \geq 0$, as $x \rightarrow \infty$ the $k^{\rm th}$ derivative of the limiting 
{\tt QuickSort} distribution function~$F$ satisfies
\begin{equation}
\label{koldright}
\exp [-x \ln x - (k \vee 1) x \ln \ln x + O(x)] 
\leq \| \Fbar^{\left(k\right)} \|_x \leq \exp [- x \ln x + O(x)].
\end{equation}
\end{theorem}

As discussed in \cite[Section~1]{janson2015tails} and in \cite[Remark 1.3(b)]{fill2019density_tails}, 
non-rigorous arguments of Knessl and Szpankowski~\cite{knessl1999quicksort} suggest very refined asymptotics, which to three logarithmic terms assert that for each $k \geq 0$ we have
\begin{equation}
\label{Fkbar asymptotics}
\Fbar^{(k)}(x) = \exp[- x \ln x - x \ln \ln x + (1 + \ln 2) x + o(x)]
\end{equation}
as $x \to \infty$ (and hence that the same asymptotics hold for $\| \Fbar^{\left(k\right)} \|_x$).  Note that for $k = 0, 1$ these expansions match the lower bounds on~$f$ and~$\Fbar$ in \refT{T:oldmain} to two logarithmic terms.

In an earlier extended-abstract version~\cite{fill2018XA} of this paper, we refined the upper bounds of \refT{T:oldmain} to match~\eqref{Fkbar asymptotics}, and we were also able to improve the lower bound in~\eqref{koldright} to match~\eqref{Fkbar asymptotics} to two terms.  Here is the main theorem of~\cite{fill2018XA}:

%\begin{theorem}[Fill and Hung~\cite[Theorem~1.2]{fill2018XA}]
\begin{theorem}[\cite{fill2018XA}, Thm.~1.2]
\label{T:XAmain}
{\rm (a)}~As $x \rightarrow \infty$, the limiting {\tt QuickSort} density function~$f$ satisfies
\begin{align}
\label{RL}
\lefteqn{\hspace{-0.5in}\exp [-x\ln x - x\ln \ln x + O(x)] \leq f(x)} \\
\label{RU} 
&\leq \exp [- x \ln x - x \ln \ln x +(1 + \ln 2) x + o(x)].
\end{align}
{\rm (b)}~Given an integer $k \geq 0$, as $x \rightarrow \infty$ the $k^{\rm th}$ derivative of the limiting 
{\tt QuickSort} distribution function~$F$ satisfies
\begin{align}
\label{kRL}
\lefteqn{\hspace{-0.5in}\exp [-x \ln x - x \ln \ln x + O(x)] \leq \| \Fbar^{\left(k\right)} \|_x} \\
\label{kRU} 
&\leq \exp [- x \ln x - x \ln \ln x + (1 + \ln 2) x + o(x)].
\end{align}
\end{theorem}

In this paper we substantially refine the upper bound
\begin{equation}
\label{RUF}
\Fbar(x) \leq \exp [- x \ln x - x \ln \ln x +(1 + \ln 2) x + o(x)]
\end{equation}
of \refT{T:XAmain}(b) with $k = 0$; we also improve the upper bounds for $k \geq 1$, though not as dramatically.  

Let
\begin{equation}
\label{Jdef}
J(t) := \int_{s = 1}^t\!\frac{2 e^s}{s}\,ds, \quad t \geq 1.
\end{equation}
It is elementary using integration by parts that $J(t)$ has the (divergent) asymptotic expansion
\begin{equation}
\label{Jexpansion}
J(t) \sim 2 t^{-1} e^t \sum_{j = 0}^{\infty} j!\,t^{-j}.
\end{equation}
Here is the main theorem of this paper:

\begin{theorem}
\label{T:newmain}
For $x \geq 2 e$, let $w \equiv w(x)$ denote the unique 
%positive real solution to
real solution satisfying $w \geq 1$ to
\begin{align*}
x &= 2 w^{-1} e^w.
\end{align*}

{\rm (a)}~As $x \rightarrow \infty$, the limiting {\tt QuickSort} distribution function~$F$ satisfies
\begin{align*}
\Fbar(x) 
&\leq \exp[- x w + J(w) - w^2 + O(\log x)] \\
&= \exp[- 2 e^w +J(w) - w^2 + O(w)].
\end{align*}

{\rm (b)}~Given an integer $k \geq 1$, as $x \rightarrow \infty$ the $k^{\rm th}$ derivative of the limiting 
{\tt QuickSort} distribution function~$F$ satisfies
\begin{equation}
\label{kRU:improvement} 
\| \Fbar^{\left(k\right)} \|_x \leq \exp [- x w + J(w) + O(\sqrt{x \log x})].
\end{equation}
\end{theorem}

\begin{remark}
\label{R:truth}
(a)~We aid the reader in gauging the approximate sizes of the various terms in the bounds appearing in \refT{T:newmain}.  It is routine to check that, 
as noted by Knessl and Szpankowski \cite[eq.~(20)]{knessl1999quicksort},
\begin{equation}
\label{wexpansion}
w = \ln(x / 2) + \ln \ln(x / 2) + (1 + o(1)) \frac{\ln \ln(x / 2)}{\ln(x / 2)} %\sim \ln x
\end{equation}
as $x \to \infty$.  Thus, by~\eqref{Jexpansion}, we have the asymptotic equivalence
\begin{equation}
\label{Jequiv}
J(w) \sim 2 w^{-1} e^w = x.
\end{equation}
From~\eqref{wexpansion}--\eqref{Jequiv} it's easy to see that \refT{T:newmain} does indeed strengthen the upper bounds in \refT{T:XAmain}.  Inclusion of the term $J(w)$ in the bounds of \refT{T:newmain} enables us effectively to bypass the entire infinite asymptotic expansion~\eqref{Jexpansion}. 

(b)~Using non-rigorous methods, Knessl and Szpankowski~\cite[see esp.\ their eq.~(18)]{knessl1999quicksort} derive the following exact asymptotics for $\Fbar(x)$ as $x \to \infty$:
\begin{align}
\lefteqn{\hspace{-.05in}\Fbar(x) =} \nonumber \\
\label{Fbar_exact} 
&\exp\left[ - x w + J(w) - w^2 - (\alpha + \mbox{$\frac{1}{2}$}) w 
- \mbox{$\frac{3}{2}$} \ln w + C - \ln(2 \sqrt{\pi}) + o(1) \right]
\end{align}
for some (unspecified) constant~$C$, with $\alpha := 2 \ln 2 + 2 \gamma - 1$, where~$\gamma$ denotes the Euler--Mascheroni constant.
Hence the bound of \refT{T:newmain}(a) on 
$\ln \Fbar(x)$ matches the conjectured asymptotics to within an additive term $O(w) = O(\log x)$.

(c)~In their notation, the non-rigorously derived eq.~(88) of~\cite{knessl1999quicksort} should read
\[
P(y) \sim \frac{C_*}{\sqrt{2 \pi}} \frac{1}{\sqrt{y}\,w_*\sqrt{1 - (1 / w_*)}} 
\exp\left[ - y w_* + \int_1^{w_*}\,\frac{2 e^u}{u}\,du - w_*^2 - \alpha w_* \right],
\]
recalling $\alpha = 2 \gamma + 2 \ln 2 - 1$.  Ignoring the factor $\sqrt{1 - (1 / w^*)}$ which 
$\sim 1$, this result in our notation is
\begin{align}
f(x) 
&\sim (2 \pi \times 2 w^{-1} e^w)^{-1/2} e^{-x w} \psi(w) \nonumber \\
\label{f_expression}
&\sim (2 \pi x)^{-1/2} \exp[ - x w + J(w) - w^2 - \alpha w - \ln w + C],
\end{align}
where~$\psi$ is the moment generating function corresponding to~$f$ and~$C$ is the same constant as at~\eqref{Fbar_exact}.  [They derive their~(88) by the ``standard saddle 
point approximation'' from the moment generating function expansion~\eqref{psi_exact} recalled in 
\refR{R:non-rigorous} below, and they derive~\eqref{Fbar_exact} by integrating~\eqref{f_expression}.]  Hence 
the bound of \refT{T:newmain}(b) on $\ln f(x)$ matches the conjectured asymptotics~\eqref{f_expression} to within an additive term $O(\sqrt{x \log x})$. 
\end{remark}

We prove 
\refT{T:newmain}
in \refS{S:proof}.  In \refS{S:LD} we use our refined asymptotic bounds on~$F$ to derive right-tail large deviation results for the distribution of the number of comparisons required by {\tt QuickSort} that sharpen somewhat the two-sided large-deviation results of McDiarmid and Hayward~\cite{mcdiarmid1996large}.

We conclude this section by repeating from~\cite{fill2018XA} an open problem concerning \emph{left}-tail behavior.  

{\bf Open Problem.\ }With 
$\Fu(x) := F(-x)$ can the \emph{lower} bounds as $x \to \infty$ in the left-tail results
\begin{align}
\label{oldleft}
\exp\left[-e^{\Gamma x + \ln \ln x +O(1)}\right] 
&\leq f(-x) \leq \exp \left[-e^{\Gamma x + O\left(1\right)}\right],\\
\label{oldkleft}
\exp\left[-e^{\Gamma x + \ln \ln x +O(1)}\right] 
&\leq \| \Fu^{\left(k\right)} \|_x \leq \exp \left[-e^{\Gamma x + O(1)}\right]
\end{align}
of~\cite{janson2015tails} and~\cite{fill2019density_tails} be improved to match the asymptotics
\[
\Fu^{(k)}(x) = \exp \left[-e^{\Gamma x + O\left(1\right)}\right]
\]
suggested by Knessl and Szpankowski~\cite{knessl1999quicksort} (and known rigorously~\cite{janson2015tails, fill2019density_tails} for \emph{upper} bounds), where 
$\Gamma := (2-\frac{1}{\ln 2})^{-1}$?

\section{Proof of the main \refT{T:newmain}}
\label{S:proof}

In \refS{S:mgf} we bound the moment generating function (mgf) $\psi$ of~$Z$.  In \refS{S:bound} we prove \refT{T:newmain}(a) by combining the Chernoff bound
\[
\Fbar(x) = \mathbb{P}(Z \geq x) \leq e^{-t x} \psi(t),
\]
for judicious choice of $t \equiv t(x) > 0$,
with our bound on~$\psi$.  In \refS{S:(b)} we prove \refT{T:newmain}(b).

\subsection{A bound on the mgf of~$Z$}
\label{S:mgf}

Let~$\psi$ denote the mgf of~$Z$.  It was shown by 
R\"osler~\cite{rosler1991limit} that~$\psi$ is everywhere finite.
In this subsection we establish a bound on $\psi(t)$ which (for large~$t$) improves on that of \cite[Lemma~2.1]{fill2018XA}, which asserts
that for every $\epsilon > 0$ there exists 
%a constant 
$a \equiv a(\epsilon) \geq 0$ such that the mgf~$\psi$ of~$Z$ satisfies
\begin{equation}
\label{XApsibound}
\psi(t) \leq \exp[(2 + \epsilon) t^{-1} e^t + a t]
\end{equation}
for every $t > 0$.  The bound~\eqref{XApsibound} in turn improved the one obtained in the proof of \cite[Lemma~6.1]{janson2015tails}, namely, that there exists $a \geq 0$ such that
\begin{equation}
\label{Jansonpsibound}
\psi(t) \leq \exp(e^t + a t) \quad \mbox{for every $t \geq 0$}.
\end{equation}

Recalling the definition~\eqref{Jdef} of $J(t)$, we next state our bound on~$\psi(t)$ which, according to~\eqref{Jexpansion}, does indeed improve on~\eqref{XApsibound} for large~$t$.

\begin{proposition}
\label{P:newerpsibound}
There exists a constant 
$a \geq 0$ such that the moment generating function~$\psi$ of~$Z$ satisfies
\begin{equation}
\label{newerpsibound}
\psi(t) \leq \exp[J(t) - t^2 + a t]
\end{equation}
for every $t \geq 1$.  
\end{proposition}

We postpone the proof of Proposition~\ref{P:newerpsibound} for a preliminary remark.

\begin{remark}
\label{R:non-rigorous}

Using non-rigorous methods, Knessl and Szpankowski~\cite{knessl1999quicksort} derive that as $t \to \infty$ the mgf~$\psi$ satisfies
\begin{equation}
\label{psi_exact}
\psi(t) = \exp[J(t) - t^2 - \alpha t - \ln t + C + o(1)], 
\end{equation}
as $t \to \infty$ for the same (unspecified) constant~$C$ as at~\eqref{Fbar_exact}, with $\alpha = 2 \ln 2 + 2 \gamma - 1$; see their equation~(71) (we have corrected a misplaced-right-parenthesis typo).  If~\eqref{psi_exact} is true, then our bound on $\ln \psi(t)$ agrees with the truth to within $O(t)$, whereas the bound~\eqref{XApsibound} (for fixed~$\epsilon$) exceeds the true value by 
$(1 + o(1)) \epsilon t^{-1} e^t$.  Thus our bound~\eqref{newerpsibound} comes substantially closer to the apparent truth than does~\eqref{XApsibound}.  
\end{remark}

The proof of Proposition~\ref{P:newerpsibound} will require the following lemma.  Recall from 
\refR{R:non-rigorous} that
$\alpha = 2 \ln 2 + 2 \gamma - 1$, and define
\[
\hpsi(t) := 
\begin{cases}
(1 - e^{-t/2}) \exp[J(t) - t^2 - \alpha t - \ln t] & \mbox{if $t > 1$} \\
1 & \mbox{otherwise}.
\end{cases}
\]

\begin{lemma}
\label{L:contranewer}
For all sufficiently large~$t$ we have the strict inequality
\[
2 \int_{u = 0}^{1 /2 }\!\hpsi(u t) \hpsi((1 - u) t) \exp[t g(u)]\,du < \hpsi(t).
\] 
\end{lemma}

\begin{proof}
Call the left side of this inequality $\lambda(t)$.  To handle $\lambda(t)$, we begin by changing the variable of integration from~$u$ to~$\eta$, where $u = \frac{1}{2} e^{-t} \eta$:
\begin{align*}
\lambda(t) 
&= e^{-t} \int_{\eta = 0}^{e^t}\!\hpsi\left(\frac{1}{2} t e^{-t} \eta\right) 
\hpsi\left(t - \frac{1}{2} t e^{-t} \eta\right) \exp\left[t g\left(\frac{1}{2}e^{-t} \eta\right)\right]\,d\eta \\
&= \int_{\eta = 0}^{e^t}\!\hpsi\left(\frac{1}{2} t e^{-t} \eta\right) 
\hpsi\left(t - \frac{1}{2} t e^{-t} \eta\right) \exp\left[2 t \phi\left(\frac{1}{2}e^{-t} \eta\right)\right]\,d\eta
\end{align*}
with $\phi(u) := u \ln u + (1 - u) \ln(1 - u) \leq 0$.

We next show that the contribution to $\int_{\eta = 0}^{e^t}$ here from $\int_{\eta = e^{t / 10}}^{e^t}$ is effectively quite negligible.  To see this, we consider the integrand in two cases.  Before breaking into cases, observe that the second argument for~$\hpsi$ is at least $t / 2$, which exceeds~$1$ if (as we may suppose) $t > 2$.  For the first case, suppose that the first argument for~$\hpsi$ also exceeds~$1$.  In this case we need to treat the sum of the $J$-values at these arguments.  But, using the increasingness of $2 s^{-1} e^s$ for $s \geq 1$, we see that if 
$a, b \geq 1$ and $a + b = t$, then
\begin{align*}
J(a) + J(b) 
&= \int_{s = 1}^a\!2 s^{-1} e^s\,ds + \int_{s = 1}^b\!2 s^{-1} e^s\,ds \\    
&\leq \int_{s = 1}^a\!2 s^{-1} e^s\,ds + \int_{s = a}^{a + b - 1}\!2 s^{-1} e^s\,ds = J(t - 1)
\end{align*}
and therefore
\begin{align*}
J(a) + J(b) - J(t) 
&\leq -[J(t) - J(t - 1)] = - \int_{s = t - 1}^t\!2 s^{-1} e^s\,ds \\ 
&\leq - 2 (t - 1)^{-1} e^{t - 1} = - (1 + o(1))\,2 e^{-1}\,t^{-1} e^t.
\end{align*}
For the second case, suppose that the first argument for~$\hpsi$ does not exceed~$1$.  In this case we need to treat $J(t - \frac{1}{2} t e^{-t} \eta) \leq J(t - \frac{1}{2} t e^{- 9 t / 10})$.  In this case, observe that
\begin{align*}
J(t - \mbox{$\frac{1}{2}$} t e^{- 9 t / 10}) - J(t) 
&\leq (\mbox{$\frac{1}{2}$} t e^{- 9 t / 10}) \cdot
- 2 (t - \mbox{$\frac{1}{2}$} t e^{- 9 t / 10})^{-1} \exp[t - \mbox{$\frac{1}{2}$} t e^{- 9 t / 10}] \\ 
&= - (1 + o(1)) e^{t / 10}.
\end{align*}
The minor contribution $\int_{\eta = e^{t/10}}^{e^t}$ is thus bounded between~$0$ and
\begin{align*}
\lefteqn{\hspace{-.4in}
(e^t - e^{t / 10}) \times \exp[J(t) - (1 + o(1)) e^{t  / 10} + O(t^2)] 
\times 1} \\
&= \exp[J(t) - (1 + o(1)) e^{t / 10} + O(t^2)] \\
&= \exp[- (1 + o(1)) e^{t / 10}] \hpsi(t).
\end{align*}

For the major contribution $\int_{\eta = 0}^{e^{t / 10}}$, we can use simple expansions for the first and third factors in the integrand, because 
$0 \leq \frac{1}{2} t e^{-t} \eta \leq \frac{1}{2} t e^{- 9 t / 10} = o(1)$:
\begin{align*}
\hpsi\left(\mbox{$\frac{1}{2}$} t e^{-t} \eta\right) 
&= 1, \\
\phi\left(\mbox{$\frac{1}{2}$} e^{-t} \eta\right) 
&= \mbox{$\frac{1}{2}$} e^{-t} \eta (-t + \ln \eta - \ln 2) - \mbox{$\frac{1}{2}$} e^{-t} \eta 
+ O(e^{-2 t} \eta^2).
\end{align*}
We also use an expansion for $J(t - \mbox{$\frac{1}{2}$} t e^{-t} \eta)$ appearing in the second factor in the integrand:
\begin{align*}
J(t - \mbox{$\frac{1}{2}$} t e^{-t} \eta) - J(t) 
&= - \mbox{$\frac{1}{2}$} t e^{-t} \eta J'(t) + \mbox{$\frac{1}{8}$} t^2 e^{-2 t} \eta^2 J''(t) 
+ O(t^2 e^{-2 t} \eta^3) \\
&= - \eta + \mbox{$\frac{1}{4}$} (t - 1) e^{-t} \eta^2 + O(t^2 e^{-2 t} \eta^3).
\end{align*}
Thus, abbreviating $t - \frac{1}{2} t e^{-t} \eta$ as $t_1 \equiv t_1(t, \eta)$, the major contribution to 
$\lambda(t)$ equals
\[
\exp[J(t)] I(t),
\]
where $I(t)$ is the integral
\begin{align*}
\lefteqn{I(t) := \int_{\eta = 0}^{e^{t / 10}}\!e^{- \eta} (1 - e^{- t_1 / 2})
\exp\big[\mbox{$\frac{1}{4}$} (t - 1) e^{-t} \eta^2 + O(t^2 e^{-2 t} \eta^3)} \\ 
&{} \qquad + t e^{-t} \eta (-t + \ln \eta - \ln 2) - t  e^{-t} \eta 
+ O(t e^{-2 t} \eta^2) - t_1^2 - \alpha t_1 - \ln t_1\big]\,d\eta.
\end{align*}
We now use the following additional expansions:
\begin{align*}
t_1^2 &= t^2 - t^2 e^{-t} \eta + O(t^2 e^{-2 t} \eta^2), \\
\ln t_1 &= \ln(t - \mbox{$\frac{1}{2}$} t e^{-t} \eta) = \ln t - \mbox{$\frac{1}{2}$} e^{-t} \eta 
+ O(e^{-2 t} \eta^2), \\
e^{- t_1 / 2} &= e^{- t / 2} [1 + \mbox{$\frac{1}{4}$} t e^{-t} \eta + O(t^2 e^{-2 t} \eta^2)].
\end{align*}
Further we can expand the factor $\exp[\cdot]$ appearing in $I(t)$ as $1 + \cdot + O(\cdot^2)$, because $\cdot = o(1)$ uniformly throughout the range of integration. 

Calculus now gives
\begin{align*}
I(t) 
&= (1 - e^{-t/2}) \exp[- t^2 - \alpha t - \ln t] \\ 
&{} \times \big[ O(t^4 e^{-2 t}) 
+ \int_{\eta = 0}^{\infty} e^{- \eta} (1 - \mbox{$\frac{1}{4}$} t e^{-3 t / 2} \eta 
+ \mbox{$\frac{1}{4}$} (t - 1) e^{-t} \eta^2 \\ 
&{} + t e^{-t} \eta (-t + \ln \eta - \ln 2)
- t e^{-t} \eta + t^2 e^{-t} \eta + \mbox{$\frac{1}{2}$} \alpha t e^{-t} \eta 
+ \mbox{$\frac{1}{2}$} e^{-t} \eta)\,d\eta \big] \\
&= (1 - e^{-t/2}) \exp[- t^2 - \alpha t - \ln t] \\ 
&{} \times [1 - \mbox{$\frac{1}{4}$} t e^{-3 t / 2} + O(t^4 e^{-2 t})].      
\end{align*}

We conclude for sufficiently large~$t$ that
\[
\lambda(t) = \hpsi(t) [1 - \mbox{$\frac{1}{4}$} t e^{-3 t / 2} + O(t^4 e^{-2 t})] < \hpsi(t).~\qed
\]
\noqed
\end{proof}

\begin{remark}
\label{R:reverse_ineq}
If we change the factor $(1 - e^{-t/2})$ in the definition of $\hpsi$ to $(1 + e^{-t/2})$, then a similar proof shows that the reverse strict inequality holds in \refL{L:contranewer}.  In fact, the proof becomes a bit simpler, since the minor contribution can simply be bounded below by~$0$. 
\end{remark}

\begin{proof}[Proof of Proposition~\ref{P:newerpsibound}]
We carry out the proof by showing that there exists $a' \geq 0$ such that 
\begin{equation}
\label{newer_psi_ineq}
\psi(t) \leq e^{a' t} \hpsi(t)
\end{equation}
for every $t > 0$.

To begin, we compare asymptotics of $\psi(t)$ and $\hpsi(t)$ as $t \to 0$.  Because~$Z$ has zero mean and finite variance, we have $\psi(t) = 1 + O(t^2)$.  On the other hand, 
$\hpsi(t) = 1$ for all $0 < t \leq 1$.  We can thus choose $t_1 > 0$ and $a'' > 0$ such that~\eqref{newer_psi_ineq} holds for $t \in [0, t_1]$ and any $a' \geq a''$.

Let $t_2 > 1$ be such that the strict inequality in \refL{L:contranewer} holds for all $t \geq t_2$, and choose~$a' \geq a''$ so that~\eqref{newer_psi_ineq} holds for $t \in [t_1, t_2]$.  Assuming for the sake of contradiction that~\eqref{newer_psi_ineq} fails for some $t > 0$, let 
$T := \inf\{t > 0:\mbox{\eqref{newer_psi_ineq} fails}\}$.  Then $T \geq t_2$, and continuity gives
\[
\psi(T) = e^{a' T} \hpsi(T).
\]
Further, if $0 < u < 1$, then~\eqref{newer_psi_ineq} holds for $t = u T$ and $t = (1 - u) T$, and thus, using our standard integral equation for~$\psi$, we have
\[
\psi(T) \leq e^{a' T} \times 2 \int_{u = 0}^{1 / 2}\!\hpsi(u T) \hpsi((1 - u) T) \exp[t g(u)]\,du,   
\]
which is strictly smaller than $e^{a' T} \hpsi(T)$ by applying \refL{L:contranewer} with $t = T \geq t_2$.  The resulting strict inequality $\psi(T) < e^{a' T} \hpsi(T)$ contradicts the definition of~$T$.  Hence~\eqref{newer_psi_ineq} holds for all $t \geq 0$.
\end{proof}

\begin{remark}
\label{R:reverse}
Using \refR{R:reverse_ineq} just as \refL{L:contranewer} is used in the proof of Proposition~\ref{P:newerpsibound}, we have the following reverse of Proposition~\ref{P:newerpsibound}:
There exists a constant 
$a \geq 0$ such that the mgf~$\psi$ of~$Z$ satisfies
\begin{equation}
\label{reversepsibound}
\psi(t) \geq \exp[J(t) - t^2 - a t]
\end{equation}
for every $t \geq 1$.
%This improves \emph{significantly} on \refR{R:psi_LB}(a)!
\end{remark}

\begin{remark}
\label{R:further}
(a)~Unfortunately, due to the need to handle small values of~$t$ in the proofs of Proposition~\ref{P:newerpsibound} and \refR{R:reverse}, we sacrifice the information in the linear term of $\ln \psi(t)$ that \refR{R:non-rigorous} and \refL{L:contranewer} strongly suggest.  Thus any further progress on asymptotic determination of $\psi$ would have to employ a technique different from the one used to derive~\eqref{Jansonpsibound}, \eqref{XApsibound}, and~\eqref{newerpsibound}.   

(b)~The extent to which we are able to make rigorous the claim~\eqref{psi_exact} and thereby, in particular, identify the linear term in $\ln \psi(t)$ is the following.  If
\[
\psi(t) = \exp[J(t) + K(t)]
\]
where we \emph{assume} $K'(t) = O(t^{b_1})$ and $K''(t) = O(t^{b_2})$ for some~$b_1$ and $b_2$ [just as we now know rigorously that $K(t) \sim - t^2 = O(t^2)$], then we must have
\[
K(t) = - t^2 - \alpha t - \ln t + C + O(t^b e^{-t})
\]
for some constant~$C$, where $b := \max\{4, 2 + 2 b_1, 2 + b_2\}$.  (Aside:\ It is natural to \emph{assume further} that $b_1 = 1$ and $b_2 = 0$, in which case $b = 4$.)
The proof of this assertion is quite similar to the proof of \refL{L:contranewer} and is omitted. 
\end{remark}

\subsection{Proof of improved asymptotic upper bound on~$\Fbar$}
\label{S:bound}

\begin{proof}[Proof of \refT{T:newmain}(a)]
Choose $t = w$, apply the Chernoff bound
\[
\Fbar(x) = \mathbb{P}(Z \geq x) \leq e^{-t x} \psi(t),
\]
and utilize Proposition~\ref{P:newerpsibound} to establish~\refT{T:newmain}(a).
\end{proof}

\begin{remark}
\label{Chernoff_remark}
(a)~For 
large~$x$, the optimal choice of~$t$ for the Chernoff bound combined with~\eqref{newerpsibound} is not $t = w$, but rather the larger $\tw \equiv \tw(x)$ of the two positive real solutions to
\[
x = 2 (\tw^{-1} e^{\tw} - \tw) + a.
\]
But the resulting improvement in the bound on $\ln \Fbar(x)$ not only is subsumed by the error bound $O(\log x)$ but in fact is asymptotically equivalent to $2 x^{-1} (\log x)^2 = o(1)$ and so is negligible even as concerns estimating $\Fbar(x)$ to within a factor $1 + o(1)$.

Here is a proof.  Use of $t = w$ vs. $t = \tw$ gives the larger expression
\[
- x w + J(w) - w^2 + a w
\]
vs.
\[
-x \tw + J(\tw) - \tw^2 + a \tw;
\]  
the increase is
\[
\Delta \equiv \Delta(x) := x (\tw - w) - [J(\tw) - J(w)] + (\tw^2 - w^2) - a (\tw - w).
\]
Using Taylor's theorem, we write
\begin{align*}
J(\tw) - J(w) 
&= 2 w^{-1} e^w (\tw - w) + t^{-1} e^t (1 - t^{-1}) (\tw - w)^2 \\
&= x (\tw - w) + (1 + o(1)) \mbox{$\frac{1}{2}$} x (\tw - w)^2
\end{align*}
where~$t$ belongs to $(w, \tw)$, and we also note
\[
\tw^2 - w^2 = 2 (\tw - w) (\tw + w) \sim 2 (\tw - w) \ln x.
\]
Thus
\[
\Delta = - (1 + o(1)) \mbox{$\frac{1}{2}$} x (\tw - w)^2 + (1 + o(1)) 2 (\tw - w) \ln x.
\]

It remains to estimate $\tw - w$.  We have
\begin{align*}
1 
&= \frac{x}{x} = \frac{2 (\tw^{-1} e^{\tw} - \tw) + a}{2 w^{-1} e^w} \\
&= \frac{w}{\tw} e^{\tw - w} - \frac{2 \tw - a}{x}.
\end{align*}
Write this as
\[
\frac{w}{\tw} e^{\tw - w} = 1 + \frac{2 \tw - a}{x}
\]
and take logs.  Note
\[
\ln\left(\frac{w}{\tw} e^{\tw - w}\right) 
= - \ln\left(1 + \frac{\tw - w}{w}\right) + \tw - w \\
\sim \tw - w
\]
and
\[
\ln\left(1 + \frac{2 \tw - a}{x}\right) \sim \frac{2 \ln x}{x}.
\]
Thus
\[
\tw - w \sim 2 x^{-1} \ln x.
\]
It now follows that
\[
\Delta = - (1 + o(1)) 2 x^{-1} (\ln x)^2 + (1 + o(1)) 4 x^{-1} (\ln x)^2
\sim 2 x^{-1} (\ln x)^2,
\]
as claimed.

(b)~If we grant the truth of~\eqref{psi_exact}, the following upper bound on~$\Fbar(x)$ resulting from use of a Chernoff inequality with $t = w$ together with~\eqref{psi_exact} still does not completely match~\eqref{Fbar_exact}:
\begin{align*}
\Fbar(x) 
&\leq \exp[- x w + J(w) - w^2 - \alpha w - \ln w + C + o(1)] \\
&= 2 \sqrt{\pi}\,w^{1/2} e^{w / 2} \times \mbox{RHS\eqref{Fbar_exact}} \sim (2 \pi x)^{1/2} \times \mbox{RHS\eqref{Fbar_exact}}.
\end{align*}
Further, use of the \emph{exactly} optimal~$t$ [ignoring the $o(1)$ remainder term in~\eqref{psi_exact}] gives a bound that is still asymptotically $(2 \pi x)^{1/2} \times \mbox{RHS\eqref{Fbar_exact}}$.
Thus if 
the asymptotic inequality 
$\Fbar(x) \leq \mbox{RHS\eqref{Fbar_exact}}$ is ever to be established rigorously, it would have to involve some technique (such as a rigorization of the saddle-point arguments used in~\cite{knessl1999quicksort}) we have not used; Chernoff bounds are insufficient.
\end{remark}

\subsection{Proof of improved asymptotic upper bounds on absolute values of derivatives of~$F$}
\label{S:(b)}

Using the improved right-tail upper bound of the distribution function in \refT{T:newmain}(a), we are now able to establish \refT{T:newmain}(b).

\begin{proof}[Proof of \refT{T:newmain}(b)]
The bound~\eqref{kRU:improvement} holds for $k = 0$ because it is cruder than the bound of \refT{T:newmain}(a).  The bound~\eqref{kRU:improvement} for general values of~$k$ then follows inductively using Proposition~6.1 of~\cite{fill2019density_tails}, according to which 
\[
\limsup_{x \to \infty} r(x)^{-1} \left( \ln \| \Fbar^{(k + 1)} \|_x -  \ln \| \Fbar^{(k)} \|_x \right) \leq 0
\]
provided $r(x) = \omega(\sqrt{x \log x})$ as $x \to \infty$.
\end{proof}
\ignore{
\begin{remark}
{\rm (a)} {\bf I don't think we can further improve the right tail upper bounds using the Landou--Kolmogorov inequality argument.} The choice of $r(x) \equiv x$ in Proposition 6.1 of~\cite{fill2019density_tails} prevents us getting more terms in $o(x)$. The natural question is if we can choose $r(x) = \log x$. In fact, if we further investigate the term $J(w_i)$, the upper bound we get here is the same as in Theorem~\ref{T:XAmain}.\\
{\rm (b)} The lower bound~\eqref{kRL:improvement} cannot be improved since we have to 
use~\eqref{RL:improvement} in our argument.
\end{remark}
}

\section{Large deviations for {\tt QuickSort}}
\label{S:LD}

With some improvements, this section repeats Section~3 of~\cite{fill2018XA}.

McDiarmid and Hayward~\cite{mcdiarmid1996large} study large deviations for the variant of {\tt QuickSort} in which the pivot (that is, the initial partitioning key) is chosen as the median of $2 t + 1$ keys chosen uniformly at random without replacement from among all the keys.  The case $t = 0$ is the classical {\tt QuickSort} algorithm of our ongoing limited focus in this paper.  Restated equivalently in terms of the random variable $Z_n$ in~\eqref{Zn} (as straightforward calculation reveals), the following is their main theorem for classical {\tt QuickSort}.

\begin{theorem}[\cite{mcdiarmid1996large}]
\label{T:McD_LD}
Let $x_n$ satisfy
\begin{equation}
\label{xnrange}
\frac{\mu_n}{n \ln n} < x_n \leq \frac{\mu_n}{n}.
\end{equation}
Then as $n \to \infty$ we have
\begin{equation}
\label{LD2sided}
\mathbb{P}(|Z_n| > x_n) = \exp\{-x_n [\ln x_n + O(\log \log \log n)]\}.
\end{equation}
\end{theorem}  
Observe that~\eqref{xnrange} is roughly equivalent to the condition that $x_n$ lie between~$2$ and 
$2 \ln n$, and rather trivially the range can be extended to $1 < x_n \leq \mu_n / n$.  But notice also that if $x_n = (\ln \ln n)^{c_n}$ with $c_n$ nondecreasing (say), then~\eqref{LD2sided} provides a nontrivial upper bound on $\mathbb{P}(|Z_n| > x_n)$ if and only if $c_n \to \infty$.

McDiarmid and Hayward require a fairly involved proof utilizing primarily the method of bounded differences pioneered by McDiarmid~\cite{mcdiarmid1989method} to establish the $\leq$ half of~\eqref{LD2sided}.  The $\geq$ half is proven by establishing (by means of another substantial argument) the right-tail lower bound
\begin{equation}
\label{LDRL}
\mathbb{P}(Z_n > x_n) \geq \exp\{-x_n [\ln x_n + O(\log \log \log n)]\},
\end{equation}
again assuming~\eqref{xnrange} (see~\cite[Lemma~2.9]{mcdiarmid1996large}).  It follows from~\eqref{LD2sided}--\eqref{LDRL} that we have the right-tail large deviation result that
\begin{equation}
\label{LDR}
\mathbb{P}(Z_n > x_n) = \exp\{-x_n [\ln x_n + O(\log \log \log n)]\}.
\end{equation}

The main point of this section [see \refT{T:LDR}(b)--(d)] is to note that~\eqref{LDR} can be refined, for deviations not allowed to be quite as large as those permitted by \refT{T:McD_LD}, rather effortlessly by combining
%the case $k = 0$ of~\refT{T:XAmain}(b)  
our upper bound [\refT{T:newmain}(a)] and lower bound [\refT{T:XAmain}(b), with $k = 0$] on the right tail of~$F$ 
with the following bound on Kolmogorov--Smirnov distance between the distributions of $Z_n$ and~$Z$ (see \cite[Section~5]{fill2002quicksort}):

\begin{lemma}[\cite{fill2002quicksort}]
\label{L:KS}
We have
\[
\sup_x |\mathbb{P}(Z_n > x) - \mathbb{P}(Z > x)| 
\leq \exp\left[- \mbox{$\frac{1}{2}$} \ln n + O\left ((\log n)^{1/2} \right)\right].
\] 
\end{lemma}

We state next our right-tail large-deviations theorem for {\tt QuickSort}.  With the additional indicated restriction on the growth of $x_n$ (which allows for $x_n$ nearly as large as 
$\frac{1}{2} \frac{\ln n}{\ln \ln n}$), parts (b)--(c) strictly refine~\eqref{LDRL} and the asymptotic upper bound on $\mathbb{P}(Z_n > x_n)$ implied by~\eqref{LDR}.  The left-hand endpoint of the interval $I_n$ in \refT{T:LDR} is chosen as $c > 1$ simply to ensure that $\sup\{- \ln \ln x:x \in I_n\} < \infty$.
\begin{theorem}
\label{T:LDR}
Let $(\omega_n)$ be any sequence diverging to $+ \infty$ as $n \to \infty$ and let $c > 1$. 
For integer $n \geq 3$, consider the interval 
$I_n := \left[c, \frac{1}{2} \frac{\ln n}{\ln \ln n}\!\left( 1 - \frac{\omega_n}{\ln \ln n} \right) \right]$.
%\smallskip
\vspace{.01in}
\par\noindent
{\rm (a)}~Uniformly for $x \in I_n$ we have
\begin{equation}
\label{moderate}
\mathbb{P}(Z_n > x) = (1 + o(1)) \mathbb{P}(Z > x) \quad \mbox{as $n \to \infty$}.
\end{equation}
{\rm (b)}~If $x_n \in I_n$ for all large~$n$, then
\begin{equation}
\label{LDRLnew}
\mathbb{P}(Z_n > x_n) \geq \exp[- x_n \ln x_n - x_n \ln \ln x_n + O(x_n)].
\end{equation}
{\rm (c)}~If $x_n \in I_n$ for all large~$n$ and $x_n \to \infty$, then
\begin{align}
\label{LDRUnewer}
\mathbb{P}(Z_n > x_n) 
&\leq \exp[- x_n w_n + J(w_n) - w_n^2 + O(\log x_n)] \\
\label{LDRUnew} 
&= \exp[- x_n \ln x_n - x_n \ln \ln x_n + (1 + \ln 2)x_n + o(x_n)],
\end{align}
where $w_n$ is the larger of the two real solutions to $x_n = 2 w_n^{-1} e^{w_n}$. \\
{\rm (d)}~If $x_n \in I_n$ for all large~$n$, then
\begin{equation}
\label{LDRnew} 
\mathbb{P}(Z_n > x_n) = \exp[- x_n \ln x_n - x_n \ln \ln x_n + O(x_n)].
\end{equation}
\end{theorem}

\begin{proof}
Parts (b)--(c) follow immediately from part~(a) and
%\refT{T:XAmain}(b),
\refT{T:newmain}(a),
and part~(d) by combining parts (b)--(c).  
So we need only prove part~(a), for which by \refL{L:KS} it is sufficient to prove that
\[
\exp\left[- \mbox{$\frac{1}{2}$} \ln n + O\left ((\log n)^{1/2} \right)\right] \leq o(\mathbb{P}(Z > x_n))
\]
with $x_n \equiv \frac{1}{2} \frac{\ln n}{\ln \ln n}\!\left( 1 - \frac{\omega_n}{\ln \ln n} \right)$; this assertion decreases in strength as 
the choice of $\omega_n$ is increased, so we may assume that $\omega_n = o(\log \log n)$.  Since, by \refT{T:XAmain}(b), we have
\[
\mathbb{P}(Z > x_n) \geq \exp[- x_n \ln x_n - x_n \ln \ln x_n + O(x_n)],
\]  
it suffices to show that for any constant $C < \infty$ we have
\[
- \mbox{$\frac{1}{2}$} \ln n + C (\ln n)^{1/2} + x_n \ln x_n + x_n \ln \ln x_n + C x_n \to - \infty.
\]
But, writing~$\L$ for $\ln$ and $\L_k$ for the $k$th iterate of~$\L$, and abbreviating $\alpha_n := 
1 - \frac{\omega_n}{\L_2 n}$, this follows from the observation that, for~$n$ large,
\begin{align*}
\lefteqn{x_n (\L x_n + \L_2 x_n + C)} \\ 
&= \frac{1}{2} \frac{\L n}{\L_2 n} \alpha_n 
[(\L_2 n - \L_3 n - \L 2 + \L \alpha_n) + \L (\L_2 n - \L_3 n - \L 2 + \L \alpha_n) + C] \\ 
&= \frac{1}{2} \frac{\L n}{\L_2 n} \alpha_n \left[ \L_2 n + C - \L 2 + \L \alpha_n
+ \L\left(1 - \frac{\L_3 n + \L 2 - \L \alpha_n}{\L_2 n}\right) \right] \\
&= \frac{1}{2} \frac{\L n}{\L_2 n} \alpha_n \left[ \L_2 n + C - \L 2 + \L \alpha_n
- (1 + o(1)) \frac{\L_3 n}{\L_2 n} \right] \\
&= \frac{1}{2} \frac{\L n}{\L_2 n} \alpha_n \left[ \L_2 n + C - \L 2 + o(1) \right] \\
&= \left( \frac{1}{2} \L n \right) \alpha_n \left[ 1 + \frac{C - \L 2 + o(1)}{\L_2 n} \right]
= \frac{1}{2} \L n - (1 + o(1)) \omega_n \frac{\L n}{2 \L_2 n}.~\qed
\end{align*}
\noqed
\end{proof}

For completeness we next present a left-tail analogue of \refT{T:LDR} [but, for brevity, only parts (b)--(c) thereof].  \refT{T:LDL} follows in similar fashion using the case $k = 0$ of~\eqref{oldkleft} in place of \refT{T:XAmain}(b).  No such left-tail large-deviation result is found in \cite{mcdiarmid1996large}.  Recall 
$\Gamma := (2-\frac{1}{\ln 2})^{-1}$ and the notation $\L_k$ used in the proof of \refT{T:LDR}.
\begin{theorem}
\label{T:LDL}
If $1 < x_n \leq \Gamma^{-1} (\L_2 n - \L_4 n - \omega_n)$ with $\omega_n \to \infty$, then
\[
\exp\left[-e^{\Gamma x_n + \L_2 x_n +O(1)}\right]
\leq \mathbb{P}(Z_n \leq - x_n)
\leq \left[-e^{\Gamma x_n + O(1)}\right].
\]
\end{theorem}

\begin{remark}
The upper bound in \refT{T:LDL} requires only the weaker restriction
\[
-M \leq x_n \leq \Gamma^{-1} (\L_2 n - \omega_n)
\]
with $M < \infty$ and $\omega_n \to \infty$.
\end{remark}

\begin{remark}
If we let $N := n + 1$ and study the slight modification $\hZ_n := (X_n - \mu_n) / N = [n / (n + 1)] Z_n$ instead of~\eqref{Zn}, then large deviation upper bounds based on tail estimates of the limiting~$F$ have broader applicability and are easier to derive, too.  The reason is that (i)~both \refT{T:newmain}(a) and the upper bound for $k = 0$ in~\eqref{oldkleft} have been derived by establishing an upper bound on the limiting mgf~$\psi$ and using a Chernoff bound, and (ii)~according to 
\cite[Theorem~7.1]{fill2002quicksort}, $\psi$ majorizes the \mgf\ $\widehat{\psi}_n$ of $\hZ_n$ for every~$n$.  It follows immediately (with~$w$ defined in the now-familiar way in terms of~$x$) that $\mathbb{P}(\hZ_n > x)$ (respectively, $\mathbb{P}(\hZ_n \leq -x)$) is bounded above uniformly in~$n$ by
\begin{align}
\lefteqn{\hspace{-.7in}\exp[- x w + J(w) - w^2 + O(\log x)]} \\
\label{cruder} 
&= \exp[-x \ln x - x \ln \ln x + (1 + \ln 2) x + o(x)]
\end{align}
(resp.,\ by $\exp \left[-e^{\gamma x + O(1)}\right]$) as $x \to \infty$; there is \emph{no restriction at all} on how large~$x$ can be in terms of~$n$. 

Here are examples of \emph{very} large values of~$x$ for which the tail probabilities are nonzero and the aforementioned bounds still match logarithmic asymptotics to lead order of magnitude, albeit not to lead-order term.  Let $\lg$ denote binary log.  The largest possible value of $X_n$ is ${n \choose 2}$ (corresponding to any binary search tree which is a path), which occurs with probability $2^{n - 1} / n!$.  The smallest possible value (supposing, for simplicity, that~$n = 2^k - 1$ for integer~$k$) is $(k - 2) 2^k + 2 = N (\lg N - 2) + 2$ (corresponding to the perfect tree, in the terminology of \cite[Section~3]{dobrow1996multiway}); according to \cite[Proposition~4.1]{dobrow1996multiway}, this value occurs with probability $\exp[-s(1) N + s(N + 1)]$, where
\[
s(\nu) := \sum_{j = 1}^{\infty} 2^{-j} \ln(2^j \nu - 1).
\]
Correspondingly, the largest possible value of $\hZ_n$ is
%\[
%\mbox{$\frac{n (n + 7)}{2 (n + 1)}$} - 2 H_n 
%= \mbox{$\frac{1}{2}$} (n + 1) - 2 \ln (n + 1) + (\mbox{$\frac{5}{2}$} - 2 \gamma) - 2 (n + 1)^{-1} 
%+ \mbox{$\frac{1}{6}$} (n + 1)^{-2} + O((n + 1)^{-4}), 
%\]
\[
\lambda_n := \mbox{$\frac{n (n + 7)}{2 (n + 1)}$} - 2 H_n 
= (1 + o(1)) \mbox{$\frac{1}{2}$} N, 
\]
and the smallest is $- \sigma_n$, with
%\[
%\sigma_n := -  2 H_{n + 1} - \lg(n + 1) - 2 
%= (2 - \mbox{$\frac{1}{\ln 2}$}) \ln(n + 1) - 2 (1 - \gamma) + (n + 1)^{-1} 
%- \mbox{$\frac{1}{6}$} (n + 1)^{-2} + O((n + 1)^{-4}).
%\]
\[
\sigma_n := -  2 H_N - \lg N - 2 
= (2 - \mbox{$\frac{1}{\ln 2}$}) \ln N + O(1).
\]

The bound~\eqref{cruder} on $\mathbb{P}(\hZ_n > \lambda_n)$ is in fact also (by the same proof) a bound on the larger probability $\mathbb{P}(\hZ_n \geq \lambda_n)$, and equals
\[
\exp\left\{- \mbox{$\frac{1}{2}$} N [\ln N + \ln \ln N - (2 \ln 2 + 1) + o(1)]\right\},
\]
whereas (using Stirling's formula) the truth is
\[
\mathbb{P}(\hZ_n \geq \lambda_n) = \exp[ - N \ln N + (1 + \ln 2) N + O(\log N)].
\]

The bound on $\mathbb{P}(\hZ_n \leq -\sigma_n)$ equals
\[
\exp\left[ - e^{\ln N + O(1)} \right] = \exp[-\Omega(N)],
\]
whereas (by \cite[Proposition~4.1 and Table~1]{dobrow1996multiway}) the truth is
\[
\mathbb{P}(\hZ_n \leq -\sigma_n) = \exp[-s(1) N + O(\log N)]
\]
and (rounded to seven decimal places) $s(1) = 0.9457553$.
\end{remark}

\bibliography{bib_file}
\bibliographystyle{plain}

\end{document}

%% file: quicksort_density_tails_def.tex
\newcommand\doi{D_{01}}

\newcommand\tw{\widetilde w}